%% file: main.tex
\documentclass[12pt,twoside,reqno]{amsart}
\usepackage[all]{xy}
        \usepackage {amssymb,latexsym,amsthm,amsmath,mathtools,dsfont,mathrsfs}
        \usepackage{enumitem,color}

\usepackage{makecell}
\usepackage{mathtools}

\usepackage{longtable}

\usepackage[backend=bibtex,backref=true,maxbibnames=999,doi=false,isbn=false,url=false
]{biblatex} 
\usepackage{float}
\usepackage{hyperref}
\long\def\symbolfootnote[#1]#2{\begingroup%
\def\thefootnote{\fnsymbol{footnote}}\footnote[#1]{#2}\endgroup}

\makeatletter
\def\imod#1{\allowbreak\mkern10mu({\operator@font mod}\,\,#1)}
\makeatother
\makeatletter
\renewcommand*\env@matrix[1][*\c@MaxMatrixCols c]{%
  \hskip -\arraycolsep
  \let\@ifnextchar\new@ifnextchar
  \array{#1}}
\makeatother
\usepackage[capitalize,nameinlink]{cleveref} 
\usepackage{comment} 
\usepackage{xcolor} 
\hypersetup{
    colorlinks,
    linkcolor={red!80!black},
    citecolor={green!80!black},
    urlcolor={blue!80!black}
}
\newtheorem{mainthm}{Theorem}

\newtheorem{theorem}{Theorem}[section]
\newtheorem{lemma}[theorem]{Lemma}

\newtheorem{proposition}[theorem]{Proposition}
\newtheorem*{theorem*}{Theorem}
\theoremstyle{definition}

\numberwithin{equation}{section}
\newcommand{\ignore}[1]{}

\newcommand{\mynote}[1]{}
\newcommand{\irr}{\operatorname{Irr}}
\newcommand{\Q}{\ensuremath{\mathbb{Q}}}
\newcommand{\GL}{\operatorname{GL}}
\newcommand{\PSL}{\operatorname{PSL}}
\newcommand{\SL}{\operatorname{SL}}
\title[Rational conjugacy classes and rational characters]{Rational conjugacy classes and rational characters for some finite simple groups}
\author[Kaur D.]{Dilpreet Kaur}
\email{dilpreetkaur@iitj.ac.in}
\address{Indian Institute of Technology Jodhpur
N.H. 62, Nagaur Road, Karwar Jodhpur 342030
Rajasthan}
\author[Panja S.]{Saikat Panja}
\email{panjasaikat300@gmail.com}
\address{Indian Statistical Institute Bangalore, Stat Math Unit, 8th Mile, Mysore Rd, RVCE Post, Bengaluru, Karnataka 560059}
\thanks{The first-named author is supported by SERB research grant through MATRICS project MTR/2022/000231. The second-named author is supported by an NBHM postdoctoral fellowship, file number ending at R\&D-II/6746.}
\date{\today}
\begin{document}
\setcounter{section}{0}
\begin{abstract}
If $G$ is a finite group, an irreducible complex-valued character $\chi$ is called rational if $\chi(g)$ is rational for all $g\in G$. Also, a conjugacy class $x^G$ is called rational, if for all irreducible complex-valued character $\chi$, the value $\chi(x^G)$ is rational. 
We prove that for $q$, a power of prime, the group $\mathrm{PSL}_2(q)$ has same number of rational characters and rational conjugacy classes. Furthermore, we verify that this equality holds for all finite simple groups whose character tables appear in the \emph{ATLAS of Finite Groups}, except for the Tits group.
\end{abstract}
\maketitle
\input{intro}
\input{prelim}
\input{results}

\input{GAP-code}

\printbibliography
\end{document}

%% file: intro.tex
\section{Introduction}\label{sec:intro}
Throughout this article we will deal with finite groups. 
It is a well known fact that the number of conjugacy classes of a finite group is same as the number of irreducible representations  over $\mathbb{C}$ of the group. 
However, there is no known bijection between the set of conjugacy classes of a group and the set of its irreducible representations in general.

To understand the correspondence between conjugacy classes of a group and its irreducible representations, group theorists study properties of both conjugacy classes and irreducible representations. 
For example, it is known that the number of \emph{real conjugacy classes} (i.e. conjugacy class of an element $g\in G$ such that $g$ is conjugate to $g^{-1}$) of a group is same as number of \emph{real characters} (i.e. characters having all values to be real) of the group over $\mathbb{C}.$ 
 Real irreducible characters are further divided into two classes, orthogonal representations and symplectic representations.
 A real irreducible character is called orthogonal, if its associated representation can be realized over the field $\mathbb{R}$ of real numbers, otherwise, it is called symplectic character. 
 In \cite{Vinroot2020}, it has been proved that for finite simple groups, the number of orthogonal representations is as same as the number of strongly real conjugacy classes.
 An element $g$ of a group $G$ is called \emph{strongly real element} if there exist $h\in G$ such that $h^2=1$ and $hgh^{-1}=g^{-1}.$
 If an element $g$ is strongly real, then its all conjugates are also strongly real elements, and its conjugacy class is called \emph{strongly real conjugacy class}. 
 However, the number of orthogonal representations does not match with number of strongly real conjugacy classes in all groups. The interested reader may see \cite{KaurKulshrestha2015}.

On similar lines, this article is an attempt to understand the connection between rational conjugacy classes and rational valued characters of a group over $\mathbb{C}.$ 
Let $G$ be a group. An element $g \in G$ is called {\it rational} if $g$ is conjugate to $g^k,$ for all $0<k<o(g)$ and $\gcd(k, o(g))=1,$ where $o(g)$ denotes the order of $g.$
An irreducible character $\chi$ is called a \emph{rational} if  $\chi(g)\in\mathbb{Q}$ for all $g\in G.$
In \cite{NavarroTiep2008}, authors show that if a group $G$ has two rational conjugacy classes, then it has two rational valued characters and vice-versa. They also conjecture that a group $G$ has three rational conjugacy classes if and only if it has three rational valued characters. 
In \cite{Rosi2018}, author proves that if a finite group has three rational conjugacy classes, then it has also three rational valued characters. He also give an example of group of order 672 and GAP Id 128, with four irreducible rational valued characters and six rational conjugacy classes. Note that this group, has a nontrivial Center of order $4$ and hence, is not simple.This example shows that the best possible generalization of the Navarro–Tiep theorem is that a group $G$ has three rational conjugacy classes if and only if it has three rational valued characters. 

Since the number of rational conjugacy classes and the number of rational characters need not match in general, it is desirable to find class of groups, for which the numbers match. In this article we explore such groups and focus on some simple groups other than the alternating groups.
More precisely, we show that number of rational conjugacy classes and the number of rational valued characters is same for $\PSL_2(q),$ where $q$  is power of a prime number. In \cite{BhuniaKaurSingh2019}, authors show that number of rational conjugacy classes and the number of rational valued characters is same for the alternating group $A_n$ for all $n\geq 5.$ In \cref{sec:sporadic}, we show that number of rational conjugacy classes does not match with number of rational valued characters for Tits group, which is a finite simple group of order $17,971,200$. 
In view of these observations, we conjecture that in case of finite simple groups of Lie type, the number of rational conjugacy classes is same as the number rational valued characters over $\mathbb{C}$. We aim to prove the following theorem;
\begin{mainthm}\label{mainthm}
    Let $G$ be the projective special linear group $\PSL_2(q)$. Then the number of rational characters and the number of rational conjugacy classes are the same. If we denote the common number by $\mathrm{RC}(G)$, then the following cases occur:
    \begin{enumerate}
        \item $\mathrm{RC}(G)=2$, whenever $q=3^{2m-1}$;
        \item $\mathrm{RC}(G)=3$, whenever (a) $q=2^m,$ or (b) $q\equiv 1\pmod{3}$ and $q\equiv3\pmod{4}$, or (c) $q\equiv 2\pmod{3}$ and $q\equiv1\pmod{4}$;                
        \item $\mathrm{RC}(G)=4$, whenever $q\equiv \pm1,\pm5\pmod{12}$ such that $q\not\equiv\pm1\pmod{24}$;
        \item $\mathrm{RC}(G)=5$, whenever (a) $q=3^{2m}$, or (b) $q\equiv \pm1\pmod{24}$;
        \item $\mathrm{RC}(G)=7$, whenever $q=p^{2m}$ where $q>13$ is odd.
    \end{enumerate}
\end{mainthm}

The proof of \cref{mainthm} follows from a case-by-case study. In \cref{sec:prep-mat} we review the character table of the group $\PSL_2(q)$, mainly to set our notations.
The \cref{sec:PSL} is devoted to the proof of \cref{mainthm}. Finally in the \cref{sec:sporadic} we list down our findings about the number of rational characters and rational conjugacy classes about the simple groups whose character tables are provided in ATLAS of finite groups \cite{CWCNP2003}, other than the alternating groups and $\PSL_2(q).$ We use the GAP Package \cite{CTblLib1.3.1} for these calculations.

%% file: prelim.tex
\section{Preparatory material}\label{sec:prep-mat}
We begin with the following result which might be well-known to experts, which relates a rational conjugacy class $C$ of a finite group $G$ to its corresponding column in the character table of $G$. 
We include a proof, which is a replica of \cite{MathStack}, for the sake of completeness.
\begin{lemma}\cite{MathStack}\label{lem:rational-conjugacy-values}
    Let $G$ be a finite group. A conjugacy class $C$ of $G$ is a rational conjugacy class if and only if for all $\chi\in \irr(G)$ one has $\chi(g)\in\mathbb{Q}$ for all $g\in C$ 
\end{lemma}
\begin{proof}
    Let $\Q(\chi)$ be a finite normal extension of $\Q$ containing the coefficients of $\varphi(g)$ for all $g\in G$, where $\varphi$ is the representation associated with $\chi$. 
    Let $m=|G|$, the order of the group and $\epsilon$ be a primitive $m$-th root of unity. 
    Further let $E$ be a finite normal extension of $\Q$ containing $\Q(\chi)$ and $\Q(\epsilon)$.
    Then there is an injective homomorphism $\psi:\mathbb{Z}_m^\times\longrightarrow \mathrm{Gal}(E,\Q)$ defined by $\psi(a)=\sigma_a$, where $\sigma_a(\epsilon)=\epsilon a$.
    Now define $\varphi_a =\tau_a\circ\varphi $ (where $\tau_a$ is the automorphism of $\GL_m(E)$ induced by applying $\sigma_a$ to the coefficients of each matrix element) and let $\chi_a$ be its character. Then $\chi_a(g)=\chi(g^a)$.

Let $c$ be conjugate to $c^n$ for $(n,m)=1$, so $\chi(c)=\chi(c^n)$. 
Since $\chi(c^n)=\chi_n(c)$, we have that $\chi(c)$ is fixed under the action of $\sigma_n$.
If this is true for every $n$ relatively prime to $m$, then $\chi(c)$ is fixed under the image of $\mathrm{Gal}(\Q(\epsilon),\Q)$. 
Thus $\chi(c)$ is a member of the fixed field of $\mathrm{Gal}(\Q(\epsilon),\Q)$, i.e. $\Q$.

On the other hand, suppose $C$ is rational and for some $\chi$ there's an $c\in C$ and $n$ relatively prime to $m$ for which $\chi(c)\neq \chi(c^n)$.
Then $\chi(c)$ is not fixed under $\sigma_n$, so it is not a member of the fixed field of $\mathrm{Gal}(\Q(\epsilon),\Q)$, a contradiction. 
So $\chi(c)=\chi(c^n)$ for all $\chi\in \irr(G)$, whence
$$\sum\limits_{i}\chi_i(c)\overline{\chi_i(c^n)}=\sum\limits_{i}\chi_i(c)^2=|Z_G(c)|$$
where the summation runs over all irreducible characters, which implies that $c$ is conjugate to $c^n$.
\end{proof}
Now we move to describing the character table of the group $\PSL_2(q)$. Before so, let us fix a few notations. 
Fix $\tau$ to be a generator of $\mathbb{F}_{q^2}^\times$ and fix $\sigma=\tau^{q+1}$, $\tau_0=\tau^{q-1}$.
Let us start with the following matrices of $\PSL_2(q)$;
\begin{align*}
    I=\begin{pmatrix}
        1 & \\ & 1
    \end{pmatrix},
      N=  \begin{pmatrix}
        1 & 1\\ & 1
    \end{pmatrix},
    N'= \begin{pmatrix}
        1 & \xi\\
         & 1
    \end{pmatrix},
\end{align*}
\begin{align*}
S(a)=\begin{pmatrix}
             \sigma^a & \\
             & \sigma^{-a}
         \end{pmatrix}, 
         T(b)=\begin{pmatrix}
            ~  & -1\\
            1 & \tau^b_0+\tau_0^{qb}
         \end{pmatrix},
\end{align*}
where $\xi\in\mathbb{F}_q^\times$ is a non-square (which exists when $q$ is odd).

First, let look at the case when $q$ is a power of $2$.
Then the conjugacy classes of $\PSL_2(q)$ have representatives as
 $I$, $N$, $S(a)$ for $1\leq a\leq \dfrac{q}{2}-1$, and $T(b)$ for $1\leq b\leq \dfrac{q}{2}$. Since there are $q+1$ many conjugacy classes, $\PSL_2(q)$ has $q+1$ irreducible representations up to equivalence. We have the representations 
$\psi_1$ and $\psi_q$ of dimensions $1$ and $q$ respectively. For each 
$1\leq k\leq \frac{q}{2}-1$ we have representations $\psi^{(k)}_{q+1}$ of dimension
$q+1$ and for each $1\leq j\leq \frac{q}{2}$, we have representations
$\psi^{(j)}_{q-1}$ of dimension $q-1$. The description of these representations 
can be found in \cite[pp. 104-105]{GeckMalle20}. We present the character table in \cref{table-psl(2)}. Note that here $\epsilon \in \mathbb{C}$ is a primitive $q-1$ root of unity and $\eta_0 \in \mathbb{C}$ is a primitive $q+1$ root of unity.
\begin{table}[H]
  \centering%
  \begin{tabular}{*{5}{c}}
    \hline
     &    &     & $1\leq a\leq \frac{q}{2}-1$ & $1\leq b\leq \frac{q}{2}$\\
    $x$& $I$ & $N$ & $S(a)$ & $T(b)$ \\
     $|x^G|$ & $1$ & $q^2-1$ & $q(q+1)$ & $q(q-1)$\\
     \hline
    $\psi_1$ & $1$ & $1$ & $1$ & $1$ \\
    $\psi_q$ & $q$ & $\cdot$ & $1$ & $-1$ \\
    $\psi_{q+1}^{(k)}$ & $q+1$ & $1$ & $\epsilon^{ak}+\epsilon^{-ak}$ &  $\cdot$\\
    $\psi_{q-1}^{(j)}$ & $q-1$ & $-1$ &  $\cdot$ & $-\eta_0^{bj}-\eta_0^{-bj}$ \\
    \hline
  \end{tabular}
\caption{Character table of $\PSL_2(2^m)$}\label{table-psl(2)}
\end{table}
Similarly when $q$ is odd, the group $\PSL_2(q)$ has $\frac{q+5}{2}$ conjugacy classes. Hence $\PSL_2(q)$ has $\frac{q+5}{2}$ irreducible representations upto equivalence. These representations are obtained by looking at non-faithful irreducible representations of $\SL_2(q)$.
In \cref{table-psl(q)14} we present the character table of $\PSL_2(q)$ in case $q\equiv 1\pmod{4}$, where we have $k=2,4,\cdots, \frac{q-5}{2}$, $j=2,4,\cdots, \frac{q-1}{2}$, $\omega=\frac{1+\sqrt{q}}{2}$ and $\omega^*=\frac{1-\sqrt{q}}{2}$.

\begin{table}[H]
  \centering%
  \begin{tabular}{*{7}{c}}
    \hline
    &&&& $1\leq a\leq \dfrac{q-5}{4}$ && $1\leq b\leq \dfrac{q-1}{4}$\\
    $x$& $I$ & $N$ & $N'$ & \makecell{$S(a)$} & $S\left(\frac{q-1}{4}\right)$ & \makecell{$T(b)$} \\
    $|x^G|$& $1$ & $\dfrac{q^2-1}{2}$ & $\dfrac{q^2-1}{2}$ & $q(q+1)$ & $\dfrac{q(q+1)}{2}$ & $q(q-1)$ \\
     \hline
    $\psi_1$ & $1$ & $1$ & $1$ & $1$ & $1$ & $1$ \\
    $\psi_q$ & $q$ & $\cdot$ & $\cdot$ & $1$ & $1$ & $-1$ \\
    $\psi_{q+1}^{(k)}$ & $q+1$ & $1$ & $1$ & $\epsilon^{ak}+\epsilon^{-ak}$ & $2\epsilon^{\frac{(q-1)k}{4}}$ & $\cdot$\\
        $\psi_{q-1}^{(j)}$ & $q-1$ & $-1$ & $-1$ & $\cdot$ & $\cdot$ & $-\eta_0^{bj}-\eta_0^{-bj}$ \\
    $\psi_+'$ & $\frac{q+1}{2}$ & $\omega$ & $\omega^*$ & $(-1)^a$ & $(-1)^{\frac{q-1}{4}}$ & $\cdot$ \\
    $\psi_+''$ & $\frac{q+1}{2}$ & $\omega^*$ & $\omega$ & $(-1)^a$ & $(-1)^{\frac{q-1}{4}}$ & $\cdot$ \\
    \hline
  \end{tabular}
  \caption{Character table of $\PSL_2(q)$, $q\equiv1\pmod{4}$}\label{table-psl(q)14}
\end{table}

In \cref{table-psl(q)34} we present the character table of $\PSL_2(q)$ in case $q\equiv 3\pmod{4}$, where we have $k=2,4,\cdots, \frac{q-3}{2}$, $j=2,4,\cdots, \frac{q-3}{2}$, $\omega=\frac{1+\sqrt{-q}}{2}$ and $\omega^*=\frac{1-\sqrt{-q}}{2}$.
\begin{table}[H]
  \centering%
  \begin{tabular}{*{7}{c}}
    \hline
    & & & & $1\leq a\leq \dfrac{q-3}{4}$ & $1\leq b\leq\dfrac{q-3}{4}$ &\\
    $x$& $I$ & $N$ & $N'$ & \makecell{$S(a)$} & \makecell{$T(b)$} & $T\left(\frac{q+1}{4}\right)$ \\
    $|x^G|$& $1$ & $\dfrac{q^2-1}{2}$ & $\dfrac{q^2-1}{2}$ & $q(q+1)$ & ${q(q-1)}$ & $\dfrac{q(q-1)}{2}$ \\
     \hline
    $\psi_1$ & $1$ & $1$ & $1$ & $1$ & $1$ & $1$ \\
    $\psi_q$ & $q$ & $\cdot$ & $\cdot$ & $1$ & $-1$ & $-1$ \\
    $\psi_{q+1}^{(k)}$ & $q+1$ & $1$ & $1$ & $\epsilon^{ak}+\epsilon^{-ak}$ & $\cdot$ & $\cdot$\\
    $\psi_{q-1}^{(j)}$ & $q-1$ & $-1$ & $-1$ & $\cdot$ & $-\eta_0^{bj}-\eta_0^{-bj}$ & $-2\eta_0^{\frac{(q+1)j}{4}}$ \\
    $\psi_-'$ & $\frac{q-1}{2}$ & $-\omega^*$ & $-\omega$ & $\cdot$ & $(-1)^{b+1}$ & $(-1)^{\frac{q+5}{4}}$ \\
    $\psi_-''$ & $\frac{q-1}{2}$ & $-\omega$ & $-\omega^*$ & $\cdot$ & $(-1)^{b+1}$ & $(-1)^{\frac{q+5}{4}}$ \\
    \hline
  \end{tabular}
      \caption{Character table of $\PSL_2(q)$, $q\equiv3\pmod{4}$}\label{table-psl(q)34}
\end{table}

Note that the entries of the character table involve the functions of the form $\cos\theta$, where $\theta$ is a rational multiple of $2\pi$.
Hence to facilitate the determination of being rational, we end this section with the following result on the rational values of the trigonometric functions.
\begin{lemma}\cite[Theorem 1]{PaolilloVincenzi21}\label{lem:values-of-trig}
    If $\theta$ is a rational multiple of $2\pi$, say $\theta = 2r\pi$ for some rational number $r$,
then the only rational values of the sine and cosine functions of $\theta$ are as follows:
\begin{align*}
    \cos(\theta), \sin(\theta) = 0, \pm\dfrac{1}{2}, \pm1.
\end{align*}
\end{lemma}

%% file: results.tex
\section{Rational classes and rational characters of $\mathrm{PSL}_2(q)$}\label{sec:PSL}
This section will be divided into two subsections; depending on the parity of the $q$ modulo $2$. We begin with the case when the $q$ is even.
\subsection{When $q\equiv 0\pmod{2}$}
If $S(a)$ is a rational conjugacy class, we need to have $\epsilon^{ak}+\epsilon^{-ak}=2\cos\left(\theta_k\right)$ has to be rational for all $1\leq k\leq \frac{q}{2}-1$, where $$\theta_k=\dfrac{2\pi\cdot ak}{q-1}.$$
Since difference between $\theta_j$ and $\theta_{j+1}$ is $\dfrac{2\pi a}{q-1}$, we must have $a=\dfrac{q-1}{3}$ if it exists.
This happens when $q$ is an {even} power of $2$. 
A somewhat different yet similar scenario occurs while examining the case of $T(b)$.
Indeed it appears that $T(b)$ is a rational conjugacy class when $3|q+1$, which occurs when $2$ is an {odd} power of $2$.

Similarly if $\psi^{(k)}_{q+1}$ is a rational character, we must have $k=\dfrac{q-1}{3}$ when it exists and $\psi_{q-1}^{(j)}$ is a rational character when $j=\dfrac{q+1}{3}$.
Hence we have the following result.
\begin{proposition}\label{prop:2}
    The group $\PSL_2(2^m)$ has $3$ rational character and $3$ rational conjugacy classes. Moreover;
    \begin{enumerate}
        \item the rational conjugacy classes of $\PSL_2(2^m)$ are conjugacy class of $I$, $N$, and either $T\left(\dfrac{q-1}{3}\right)$ when $m$ is even or $S\left(\dfrac{q+1}{3}\right)$ otherwise.
        \item the rational characters of $\PSL_2(2^m)$ are $\psi_1$, $\psi_q$, and either $\psi_{q+1}^{(q-1)/3}$ when $m$ is even or $\psi_{q-1}^{(q+1)/3}$ otherwise.
    \end{enumerate}
\end{proposition}
\subsection{When $q\equiv 1\pmod{2}$} We first discuss results parallel to those discussed in the previous subsection, keeping two cases in mind $q\equiv 3\pmod{4}$ and $q\equiv 1\pmod{4}$.
\subsubsection{When $q\equiv3\pmod{4}$} When $q\equiv 3\pmod{4}$, for $1\leq \dfrac{k}{2}\leq \dfrac{q-3}{4}$ to have $\psi_{q+1}^{(k)}$ as a rational character we must have for all $1\leq a\leq \dfrac{q-3}{4}$, the quantity
\begin{align*}
    \epsilon^{ak}+\epsilon^{-ak}=2\cos\left(\dfrac{2\pi ak}{q-1}\right)
\end{align*}
to be rational.
Letting $\theta_a=\dfrac{2\pi ak}{q-1}$, we see that $\theta_a-\theta_{a-1}=\dfrac{2\pi k}{q-1}$.
By using \cref{lem:rational-conjugacy-values} we see that $\dfrac{2k}{q-1}\in\left\{\dfrac{1}{3},\dfrac{1}{2},\dfrac{2}{3}\right\}$, since $2k\leq q-1$.

Now for $1\leq \dfrac{j}{2}\leq \dfrac{q-3}{4}$ to have $\psi_{q-1}^{(j)}$ to be rational, by the above observation, it is enough to have all $T(b)$ for $1\leq b \leq \dfrac{q-3}{4}$ are rational valued.
By a similar argument as above this happens if and only if $\dfrac{2j}{q+1}\in\left\{\dfrac{1}{2},\dfrac{2}{3},\dfrac{1}{3}\right\}$.

Again $S(a)$ is a rational class if and only if for all $k$ satisfying $1\leq \dfrac{k}{2}\leq \dfrac{q-3}{4}$ the values
\begin{align*}
    \epsilon^{ak}+\epsilon^{-ak}=2\cos\left(\dfrac{2\pi ak}{q-1}\right)
\end{align*}
are rational . Letting $\theta_k=\dfrac{2\pi ak}{q-1}$, we have $\theta_k-\theta_{k-2}=\dfrac{4\pi a}{q-1}$.
Again, by using \cref{lem:values-of-trig}, we have that $\dfrac{4a}{q-1}\in\left\{\dfrac{1}{3},\dfrac{1}{2},\dfrac{2}{3}\right\}$, since $2a<q-1$.

Now for $1\leq b\leq \dfrac{q-3}{4}$ the class $T(b)$ is a rational class if and only if 
\begin{align*}
    \eta_0^{bj}+\eta_0^{-bj}=2\cos\left(\dfrac{2\pi bj}{q+1}\right)
\end{align*}
is rational for all $\theta_j=\dfrac{2\pi bj}{q+1}$, and by a similar argument as above this happens if and only if $\dfrac{4b}{q+1}\in\left\{\dfrac{1}{3},\dfrac{1}{2},\dfrac{2}{3}\right\}$. 

The class $T\left(\dfrac{q+1}{4}\right)$ is rational if and only if for all $j$ the value $$\cos\left(\dfrac{2 \pi }{(q+1)}\cdot\dfrac{(q+1)j}{4}\right)=\cos\left(\dfrac{j\pi}{2}\right)$$ is rational, which is indeed the case.

\subsubsection{When $q\equiv 1\pmod{4}$} Let us quickly describe this case.
As the reasons are almost similar, we mention them without a proof; 
    (a) For $1\leq \dfrac{k}{2}\leq\dfrac{q-5}{4}$, the character $\psi^{(k)}_{q+1}$ is rational if and only if $\dfrac{2k}{q-1}\in\left\{\dfrac{1}{3},\dfrac{1}{2},\dfrac{2}{3}\right\}$, 
    (b) For $1\leq\dfrac{j}{2}\leq\dfrac{q-1}{4}$, the character $\psi^{(j)}_{q-1}$ for $1\leq j\leq \dfrac{q-1}{4}$ is rational if and only if $\dfrac{2j}{q+1}\in\left\{\dfrac{1}{3},\dfrac{1}{2},\dfrac{2}{3}\right\}$,
    (c) The class $S\left(\dfrac{q-1}{4}\right)$ is a rational class since $\cos\left(\dfrac{2\pi}{q-1}\cdot\dfrac{(q-1)k}{4}\right)=\cos\left(\dfrac{\pi k}{4}\right)$ is always rational,
    (d) The class $S(a)$ for $1\leq a\leq \dfrac{q-5}{4}$ is rational if and only if $\dfrac{4a}{q-1}\in\left\{\dfrac{1}{3},\dfrac{1}{2},\dfrac{2}{3}\right\}$,
    (e) The class $T(b)$ for $1\leq b\leq \dfrac{q-1}{4}$ is rational if and only if $\dfrac{4b}{q+1}\in\left\{\dfrac{1}{3},\dfrac{1}{2},\dfrac{2}{3}\right\}$
We first deal with the cases when $q$ is a power of $3$. We have two cases to consider, the first being $q=3^{2\ell}$ and the other being $q=3^{2\ell+1}$. We have the following
\begin{lemma}\label{prop:case-3}
    Let $q=3^t$ for some integer $t\geq 1$. If $t$ is odd then the classes $I$ and $T\left(\dfrac{q+1}{4}\right)$; the characters $\psi_1$ and $\psi_q$ are rational . When $t$ is even the classes  $I$, $N$, $N'$, $S\left(\dfrac{q-1}{4}\right)$, and $S\left(\dfrac{q-1}{8}\right)$; and the characters $\psi_1$, $\psi_q$, $\psi_-',\psi_-''$ and $\psi^{(k)}_{q+1}$ for $k=\dfrac{q-1}{4}$ are rational .
\end{lemma}
\begin{proof}
    First let $q=3^{2\ell+1}$ for some integer $\ell\geq 0$, in which case $q\equiv 3\pmod{4}$. 
    Clearly the classes $N$ and $N'$ are not rational .
    By the discussion above we have the classes $I$ and $T\left(\dfrac{q+1}{4}\right)$ to be rational . 
    So, if any other rational classes exists, it has to be of the form $S(a)$ or $T(b)$.
    Now if such an $a$ exists, we must have $q\in\{6a+1,8a+1,12a+1\}$.
    This can never happen and hence no such $S(a)$ is a rational class. Similar argument works for $T(b)$ as well.
    Let us now quickly see the cases for rational characters. 
    We have that $\psi_1$ and $\psi_q$ are rational characters, and $\psi_+', \psi_{+}''$ are non-rational characters.
    If $\psi^{(k)}_{q+1}$ is a rational character for some $k$ then $q\in\{6k+1,4k+1,3k+1\}$ which is never possible. Similarly no $\psi_{q-1}^{(j)}$ is a rational character.
    
    Next let $q=3^{2\ell}$ for some integer $\ell\geq 1$. 
    In this case $q\equiv 1\pmod{4}$.
    We have the classes of $I$, $N$, $N'$ and $S\left(\dfrac{q-1}{4}\right)$ to be rational , and the characters 
    $\psi_1$, $\psi_q$, $\psi_-'$, $\psi_-''$ to be rational .
    If any $S(a)$ is rational  then $q\in\left\{6a+1,8a+1,12a+1\right\}$. Such a unique $a$ exists; $a=\dfrac{q-1}{8}$, and this also implies that $\psi^{(k)}_{q+1}$ is also rational  for $\dfrac{k}{2}=\cdot\dfrac{q-1}{8}$. 
    Furthermore, since $q\equiv 1\pmod{4}$, none of the classes $T(b)$ and the characters $\psi_{q-1}^{(j)}$ are rational . This finishes the proof.
\end{proof}
Now we are ready to prove the main theorem of the article.
 We have the following observations:
\color{black}
Using the observations above we prove a couple of propositions which will lead us to a proof of the main theorem.
\color{black}
\begin{proposition}\label{prop:evenpower}
    Let $p>3$ be a prime, and $q=p^{2\ell}\geq 13$. Then the group $\PSL_2(q)$ has seven rational conjugacy class and seven rational characters.
\end{proposition}
\begin{proof}
    Note that in this case $q\equiv 1\pmod{4}$. 
    Hence the classes given by $I$, $N$, $N'$ and $S\left(\dfrac{q-1}{4}\right)$ are rational , whereas the characters given by $\psi_1$, $\psi_q$, $\psi_+'$ and $\psi_+''$ are rational .
    It remains to prove that there are exactly three more rational classes and three more rational characters.
    {When $q=p^{2\ell}$ and $q\not\equiv 0\pmod{3}$, we have that
    \begin{align*}
        q-1\equiv0\pmod{4},\qquad q-1\equiv 0\pmod{3},\qquad\text{and}\qquad q-1\equiv0\pmod{8}.
    \end{align*}
    Hence if we choose $a\in\left\{\dfrac{q-1}{6},\dfrac{q-1}{8},\dfrac{q-1}{12}\right\}$ then the classes $S(a)$ are rational .}

    Furthermore for $\dfrac{k}{2}\in\left\{\dfrac{q-1}{12},\dfrac{q-1}{8},\dfrac{q-1}{6}\right\}$ the characters $\psi_{q+1}^{(k)}$ are rational . Note that $q+1\not\equiv 0\pmod{3}$ and $q+1\not\equiv 0\pmod{4}$, which implies that none of the other classes are rational , and the same applies for othe characters as well.
\end{proof}
\begin{proposition}\label{q+-1-24}
            Let $q$ be an odd power of prime. Then the group $\PSL_2(q)$ has five rational classes and five rational characters, whenever $q\equiv \pm1\pmod{24}$.
\end{proposition}
\begin{proof}
    We have two cases to consider. Firstly consider $q\equiv 1\pmod{24}$, which also implies that $q\equiv 1\pmod{4}$, the classes $I$ and $S\left(\dfrac{q-1}{4}\right)$ are rational; whereas the characters $\psi_1$ and $\psi_q$ are rational.
    
    Note that none of the classes $T(b)$ are rational, as that would imply $q\in\{12b-1,8b-1,6b-1\}$ and also none of the characters $\psi^{(j)}_{q-1}$ are rational.
    Since $q\equiv 1\pmod{24}$, the quantities $\dfrac{q-1}{12}$, $\dfrac{q-1}{8}$, $\dfrac{q-1}{6}$ are integer, and hence the classes $S\left(\dfrac{q-1}{12}\right)$, $S\left(\dfrac{q-1}{8}\right)$ and, $S\left(\dfrac{q-1}{6}\right)$ rational.
    A similar argument implies that the characters $\psi_{q+1}^{(k)}$ for $k\in\left\{\dfrac{q-1}{6},\dfrac{q-1}{4},\dfrac{q-1}{3}\right\}$ are rational.

    Next let consider $q\equiv -1\pmod{24}$ ,which also implies that $q\equiv 3\pmod{4}$). Hence the classes $I$ and $T\left(\dfrac{q+1}{4}\right)$; and the characters $\psi_1$ and $\psi_q$ are rational.
    
    Hence none of the classes $S(a)$ are rational, as that would imply $q\in\{12a+1,8a+1,6a+1\}$ and also none of the characters $\psi^{(k)}_{q+1}$ are rational.
    Since $q\equiv -1\pmod{24}$, the quantities $\dfrac{q+1}{12}$, $\dfrac{q+1}{8}$, $\dfrac{q+1}{6}$ are integer, and hence the classes $S\left(\dfrac{q+1}{12}\right)$, $S\left(\dfrac{q+1}{8}\right)$ and, $S\left(\dfrac{q+1}{6}\right)$ rational.
    A similar argument implies that the characters $\psi_{q-1}^{(j)}$ for $k\in\left\{\dfrac{q+1}{6},\dfrac{q+1}{4},\dfrac{q+1}{3}\right\}$ are rational.
\end{proof}
\begin{proposition}\label{q+-1-12}
            Let $q$ be an odd power of an odd prime such that $q\not\equiv\pm1\pmod{24}$. Then the group $\PSL_2(q)$ has four rational classes and four rational characters, whenever $q\equiv \pm1\pmod{12}$.
\end{proposition}
\begin{proof}
First let us assume $q\equiv 1\pmod{12}$, which automatically implies that $q\equiv 1\pmod{4}$. Hence the classes $I$, $S\left(\dfrac{q-1}{4}\right)$; and the characters $\psi_1$, $\psi_q$ are rational. The classes $S\left(\dfrac{q-1}{12}\right)$, $S\left(\dfrac{q-1}{6}\right)$ are rational, and the characters $\psi_{q+1}^{(k)}$ for $k=\dfrac{q-1}{6},\dfrac{q-1}{3}$ are rational. Note that $8\nmid q-1$, since $q-1\neq 0\pmod{24}$. No other characters and classes are rational.

We quickly note down the real class and real character for the case $q\equiv-1\pmod{24}$. In this case the classes $I$, $T\left(\dfrac{q+1}{4}\right)$, $T\left(\dfrac{q+1}{6}\right)$, and $T\left(\dfrac{q+1}{12}\right)$ are real. Furthermore in this case, the characters $\psi_1$, $\psi_q$ and $\psi_{q-1}^{(j)}$ for $j=\dfrac{q+1}{6}, \dfrac{q+1}{3}$ are rational.
\end{proof}
\begin{proposition}\label{q+-5-12}
    Let $q$ be an odd power of an odd prime, such that $q\not\equiv\pm1\pmod{24}$. Then the group $\PSL_2(q)$ has four rational classes and four rational characters, when $q\equiv \pm5\pmod{12}$.
\end{proposition}
\begin{proof}
    When $q\equiv 5\pmod{12}$, we have $q\equiv1\pmod{4}$ and hence the classes $I$, $S\left(\dfrac{q-1}{4}\right)$ are real. Moreover the class $T\left(\dfrac{q+1}{6}\right)$ and $S\left(\dfrac{q-1}{8}\right)$ are real. 
    In this case the characters $\psi_1$, $\psi_q$, $\psi_{q+1}^{(k)}$ for $k=\dfrac{q-1}{4}$ and the character $\psi^{(j)}_{q-1}$ for $j=\dfrac{q+1}{3}$ are rational.

    Since the proof is similar, we note down the real classes and real characters in the case $q\equiv -5\pmod{12}$. The real classes are given by $I$, $T\left(\dfrac{q+1}{4}\right)$, $S\left(\dfrac{q-1}{6}\right)$ and $T\left(\dfrac{q+1}{8}\right)$. The characters given by $\psi_1$, $\psi_q$, $\psi^{(k)}_{q+1}$ for $k=\dfrac{q-1}{3}$ and the character $\psi_{q-1}^{(j)}$ for $j=\dfrac{q+1}{4}$ are rational.
\end{proof}
\begin{proposition}\label{q-1-3-4}
            Let $q$ be an odd power of an odd prime such that $q\equiv\pm7\pmod{24}$. Then the group $\PSL_2(q)$ has three rational classes and three rational characters.
\end{proposition}
\begin{proof}
   If one has $q\equiv7\pmod{24}$, then automatically $ q\equiv 3\pmod{4}$, we have the classes $I$, $T\left(\dfrac{q+1}{4}\right)$; and the characters $\psi_1$, $\psi_q$ are rational. In this case we further have that $\dfrac{q-1}{6}$ is an integer and hence the class $S\left(\dfrac{q-1}{6}\right)$, and the character $\psi_{q-1}^{(k)}$ is rational for $k=\dfrac{q-1}{3}$. Note that no other class of the form $S(a)$, $T(b)$ and no other character of the form $\psi_{(q+1)}^{(k)}$ and $\psi_{(q-1)}^{(j)}$ are rational.

   In the next case the classes $I$, $S\left(\dfrac{q-1}{4}\right)$ and $T\left(\dfrac{q+1}{6}\right)$ are rational; moreover the characters $\psi_1$, $\psi_q$ and $\psi_{q-1}^{(j)}$ for $j=\dfrac{q+1}{3}$ are rational. 
\end{proof}
\begin{proof}[Proof of \cref{mainthm}]
    When $q$ is a power of $2$, the result follows from \cref{prop:2}, and when $q$ is a power of $3$, the result follows from \cref{prop:case-3}. Here after assume that $q$ is power of a prime $p>3$.

    When $q=p^{2\ell}$ then the result follows from \cref{prop:evenpower}. Next we consider values of $q$ modulo $24$. Since $q$ is a power of prime, these values are from the set $\{\pm 1, \pm 5, \pm 7\}$. When $q\equiv\pm1\pmod{24}$, the result follows from \cref{q+-1-24} and \cref{q+-1-12}, considering the fact that in this case $q\equiv\pm1\pmod{12}$ as well. 
    The case of $q\equiv\pm5\pmod{24}$ is considered in \cref{q+-5-12}. The analysis for the other two cases is discussed in \cref{q-1-3-4}. This finishes the proof

\end{proof}
\color{black}

%% file: GAP-code.tex
\section{Rational conjugacy classes and rational valued characters\\ for ATLAS groups}\label{sec:sporadic}
Using computer algebra software GAP \cite{GAP4} and GAP Package \cite{CTblLib1.3.1}, we computed that the number of rational conjugacy classes and the number of rational valued characters for all simple groups, whose character tables are mentioned in the ATLAS of finite groups \cite{CWCNP2003}, except the Alternating groups and $\PSL_2(q).$
In the following table, we use the standard notations for these Atlas groups, and 
we denote the number of rational conjugacy classes in a group with $c$, and the number of rational valued characters with $r$. From this table, we notice that the number of rational conjugacy classes is same as the number of rational valued characters except the Tits groups. For the Tits Group, the number of rational conjugacy classes is $12$ and the number of rational valued characters is $10.$

\begin{center}
    \begin{longtable}{|c|c|c||c|c|c||c|c|c|}
    \hline
     \text{Group}     &  $c$  & $r$ & \text{Group}     &  $c$  & $r$ & \text{Group}     &  $c$  & $r$\\
     \hline
    $\mathrm{M}_{11}$  &  6 & 6  & $\mathrm{Co}_1$ & 97 & 97 & $\mathrm{Ru}$ & 22 &  22  \\
   \hline
    $\mathrm{M}_{12}$  &  13 & 13 &  $\mathrm{Co}_2$ & 52 & 52 &  $\mathrm{Suz}$ & 31 & 31 \\
   \hline
    $\mathrm{M}_{22}$  &  8 & 8  & $\mathrm{Co}_3$ & 34 & 34 &  $\mathrm{O^\prime N}$  & 15 & 15\\
   \hline
    $\mathrm{M}_{23}$  &  7 & 7  &  $\mathrm{Fi}_{22}$ & 55 & 55 & $\mathrm{HN}$ & 32 & 32\\
   \hline
    $\mathrm{M}_{24}$  &  16 & 16  & $\mathrm{Fi}_{23}$ & 86 & 86 & $\mathrm{Ly}$ & 29 & 29 \\
   \hline
    $\mathrm{J}_{1}$  &  6 & 6  & $\mathrm{Fi}_{24}$ & 167 & 167 &  $\mathrm{Th}$ & 30 & 30\\
   \hline
     $\mathrm{J}_{2}$  &  11 & 11 & $\mathrm{HS}$  &  20 & 20 & $\mathrm{B}$  & 166 & 166\\
   \hline
   \text{Group}     &  $c$  & $r$ & \text{Group}     &  $c$  & $r$ & \text{Group}     &  $c$  & $r$\\
     \hline
     $\mathrm{J}_{3}$  &  8 &  8 & $\mathrm{McL}$ & 12 & 12 & $\mathrm{M}$ & 150 & 150 \\
   \hline
   $\mathrm{J}_4$ & 29 & 29 &  $\mathrm{He}$ & 17 & 17 &   {\bf $\mathrm{T}$} & {\bf 12} & {\bf 10} \\
   \hline
   $\PSL_3(3)$ & 6 & 6 & $\mathrm{PSU}_3(3)$ & 6 & 6 & $\mathrm{PSL}_3(4)$ & 6 & 6\\
   \hline
   $\mathrm{PSU}_4(2)$ & 10 & 10 & $\mathrm{PSL}_3(5)$ & 8 & 8 & $\mathrm{PSp}_4(4)$ & 11 & 11 \\
    \hline
   $\mathrm{PSU}_3(4)$ & 4 & 4 & $\mathrm{PSU}_3(5)$ & 10 & 10 & $\mathrm{Sz}(8)$ & 3 & 3 \\
 \hline
   $\mathrm{PSp}_6(2)$ & 30 & 30 & $\mathrm{PSL}_3(7)$ & 10 & 10 & $\mathrm{PSU}_4(3)$ & 14 & 14 \\
    \hline
   $\mathrm{G}_2(3)$ & 19 & 19 & $\mathrm{PSp}_4(5)$ & 17 & 17 & $\mathrm{PSU}_3(8)$ & 6 & 6 \\
    \hline
   $\mathrm{PSU}_3(7)$ & 8 & 8 & $\mathrm{PSL}_4(3)$ & 23 & 23 & $\mathrm{PSL}_5(2)$ & 13 & 13 \\
    \hline
   $\mathrm{PSU}_5(2)$ & 13 & 13 & $\mathrm{PSL}_3(8)$ & 4 & 4 & $\mathrm{Sz}(32)$ & 3 & 3 \\
    \hline
   $\mathrm{PSL}_3(9)$ & 6 & 6 & $\mathrm{PSU}_3(9)$ & 6 & 6 & $\mathrm{PSU}_3(11)$ & 10 & 10 \\
    \hline
   $\mathrm{P}\Omega^+_8(2)$ & 53 & 53 & $\mathrm{P}\Omega^-_8(2)$ & 27 & 27 & $^3\mathrm{D}_4(2)$ & 14 & 14 \\
   \hline
   $\mathrm{PSL}_3(11)$ & 8 & 8 & $\mathrm{G}_2(4)$ & 14 & 14 & $\mathrm{P}\Omega_7(3)$ & 54 & 54 \\
    \hline
   $\mathrm{PSp}_6(3)$ & 24 & 24 & $\mathrm{G}_2(5)$ & 31 & 31 & $\mathrm{PSU}_6(2)$ & 34 & 34 \\
   \hline
   $\mathrm{R}(27)$ & 5 & 5 & $\mathrm{PSp}_8(2)$ & 79 & 79 & $\mathrm{P}\Omega^+_8(3)$ & 112 & 112 \\
   \hline

    $\mathrm{P}\Omega^-_8(3)$ & 82 & 82 & $\mathrm{P}\Omega^+_{10}(2)$ & 69 & 69 & $\mathrm{P}\Omega^-_{10}(2)$ & 69 & 69\\
     \hline
   $\mathrm{F}_4(2)$ & 89 & 89 & $^2\mathrm{E}_6(2)$ & 108 & 108 & $\mathrm{E}_6(2)$ & 80 & 80 \\
    \hline
    \caption{List of number of rational conjugacy classes and rational valued characters for finite simple groups from the ATLAS }
    \end{longtable}
\end{center}